\documentclass[12pt]{article}

\usepackage{amsmath}
\usepackage{latexsym}
\usepackage{amsfonts}
\usepackage{amssymb}
\usepackage{amscd}
\usepackage{theorem}
\usepackage{a4wide}
\usepackage{color}
\usepackage[utf8]{inputenc}
\usepackage[title]{appendix}
\usepackage{manyfoot}

\theoremstyle{thmstyleone}
\newtheorem{theorem}{Theorem}
\newtheorem{lemma}[theorem]{Lemma}
\newtheorem{proposition}[theorem]{Proposition}
\newtheorem{corollary}[theorem]{Corollary}

\theoremstyle{thmstyletwo}%
\newtheorem{example}[theorem]{Example}%

\theoremstyle{thmstylethree}%
\newtheorem{definition}{Definition}%

\newenvironment{proof}{    
	\noindent
	\textbf{Proof.}}{
	\hfill $\Box$
	\vspace{3mm}
}

\numberwithin{equation}{section}

\newcommand{\N}{\mathbb{N}} 
\newcommand{\C}{\mathbb{C}} 
\newcommand{\D}{\mathbb{D}} 
\newcommand{\T}{\mathbb{T}}

\newcommand{\eps}{\varepsilon}

\newcommand{\vgd}{v_{\gamma+\delta}}
\newcommand{\vg}{v_{\gamma}}

\title{Generalized Hilbert operators acting on weighted spaces of holomorphic functions with sup-norms}

\author{María J. Beltrán-Meneu\footnote{Departament de Matemàtiques, Universitat Jaume I, Av.Vicent Sos Baynat, s/n, Castell\' o de la Plana, E-12071 (Spain), mmeneu@uji.es,   https://orcid.org/0000-0001-9645-096X.}, 
	José Bonet\footnote{Instituto Universitario de Matem\'atica Pura y Aplicada IUMPA, Universitat Politècnica de València, Camino de Vera, s/n, Valencia,  E-46022 (Spain), jbonet@mat.upv.es, https://orcid.org/0000-0002-9096-6380.} 
	and Enrique Jordá\footnote{Instituto Universitario de Matem\'atica Pura y Aplicada IUMPA, Universitat Politècnica de València, Camino de Vera, s/n, Valencia,  E-46022 (Spain), ejorda@mat.upv.es, https://orcid.org/0000-0003-2980-1699.}}
\date{}

\begin{document}
	
	\maketitle
	
	\begin{abstract}
		
The behaviour of the generalized Hilbert operator associated with a positive finite Borel measure $\mu$ on $[0,1)$ is investigated when it acts on weighted Banach spaces of holomorphic functions on the unit disc defined by sup-norms and on Korenblum type growth Banach spaces. It is studied when the operator is well defined, bounded and compact. To this aim, we study when it can be represented as an integral operator. We observe  important differences with the  behaviour of the  Cesàro-type operator acting on these spaces, getting that boundedness and compactness are equivalent concepts for some standard weights. For the space of bounded holomorphic functions on the disc   and for the Wiener algebra, we get also this equivalence, which is characterized in turn by the summability of the moments of the measure $\mu.$ In the latter case, it is also equivalent to nuclearity. 	Nuclearity of the generalized Hilbert operator acting on related spaces, such as  the classical Hardy space, is also analyzed.

\end{abstract}
	
		

	\section{Introduction}
	
	Let $H(\D)$ be the space of holomorphic functions on the unit disc $\D$	of the complex plane $\C$ endowed with the topology $\tau_0$ of uniform convergence on the compact subsets of $\D$. If  $\mu$ is a positive finite Borel measure on $[0,1)$, we denote by $\mu_n:=\int_{0}^{1}t^nd\mu(t)$ the {\it $n$-th moment of $\mu$}, $n\in \N_0$. Here $\N_0=\N \cup \{0\} = \{0,1,2,...\}$. For
$f\in H(\D),$ $f(z)=\sum_{n=0}^{\infty}a_nz^n,$ $z\in \D,$ we consider its image  by the generalized Hilbert operator
	$$H_{\mu}(f)(z):=\sum_{n=0}^{\infty}\left(\sum_{k=0}^{\infty}\mu_{n+k}a_k\right)z^n, \ z\in \D,$$
	in case it defines an analytic function in $\D$.  	The operator $H_{\mu}$ is a generalization of the classical Hilbert operator $H$. Indeed, if $\mu$ is the Lebesgue measure, then 	$H_\mu=H$ and $\mu_n=\frac{1}{n+1}, n\in \N_0.$
	
	The following related operator will also be considered:
	$$ I_\mu(f)(z):=\int_{0}^{1}\frac{f(t)}{1-tz}d\mu(t),$$
	whenever it defines an analytic function in $\D$ for $f \in H(\D)$.  	  It is well known that, if $\int_{0}^{1}\frac{d\mu(t)}{1-t}<\infty,$ i.e., $\sum_{n=0}^{\infty}\mu_n<\infty,$ then $I_{\mu}$ is well defined and continuous on $H(\D).$
	
	These two operators coincide on the polynomials. Therefore, if  they are continuous on a  space containing the polynomials as a  dense subset, then they must coincide. 

The Hilbert operator and the generalized Hilbert operators have received much attention in the last decades and their behaviour have been investigated when acting on different spaces of holomorphic functions on the disc. We refer the reader to \cite{AMS, Dai, DJV, LMN} and the references therein for the Hilbert matrix and to \cite{CGP, GGPS, GP, GirelaMerchan, GM2, Jevtic2019, PR, Ye} for the generalized Hilbert and other Hankel operators.  

The purpose of our article is to investigate the boundedness and compactness of the generalized Hilbert operator $H_{\mu}$ and its companion operator $I_{\mu}$  when they act on weighted spaces of holomorphic functions on $\D$ defined by sup-norms, and in particular on the Korenblum type growth Banach spaces $H_{\vg}^{\infty}$ and  $H_{\vg}^0$, $\gamma>0$. When $p<1$, the results that we present for the continuity and compactness of $I_\mu:H_{\vg}^{\infty}\to H_{v_p}^{\infty}$ are analogous to those obtained for the Ces\`aro-type operator in \cite{BBJCesaro}. When $p\geq 1$ the situation is completely different. In particular, we show that  the continuity of $I_\mu:H_{v_1}^{\infty}\to H_{v_1}^{\infty}$ is equivalent to the compactness of the   operator. Both conditions are equivalent to the  above mentioned summability of the sequence  of moments $(\mu_n)_n$ associated to $\mu$. In the last section we show that the summability condition is in fact equivalent to the continuity and also to the compactness of $I_\mu$ when it acts on the space of bounded holomorphic functions on the disc $H^\infty$ and in the Wiener algebra $A(\T)$. We get $I_\mu=H_\mu$ in these cases. Moreover, we show that the continuity of $I_\mu:A(\T)\rightarrow A(\T)$ is equivalent not only to the compactness but also to  nuclearity. The equivalence of the summability of  $(\mu_n)_n$ and the continuity of $I_\mu=H_\mu\in \mathcal{L}(H^\infty)$ was obtained recently  by Girela and Merch\'an  in \cite{GirelaMerchan}.

We present some necessary definitions. A {\it weight} $v$ on $\D$  is  a radial ($v(z)=v(|z|),\ z\in \D$), continuous and  non-increasing function with respect to $|z|$ satisfying $\lim_{r\rightarrow 1^-}v(r)=0$. Given a weight $v$, we define the {\it weighted Banach spaces of holomorphic funcions }

$$H_v^{\infty}:=\{f\in H(\D): \ \sup_{z\in \D}v(z)|f(z)|<\infty\}$$
and

$$H_v^{0}:=\{f\in H_v^{\infty}: \ \lim_{|z|\rightarrow 1^-} v(z)|f(z)|=0\}.$$

Both are Banach spaces under the norm $\|f\|_v:=\sup_{z\in \D}v(z)|f(z)|, \ f\in H_v^{\infty}.$ They appear in the study of growth conditions of analytic functions and have been investigated by several authors since the work of Shields and Williams (see eg.  \cite{BBG}, \cite{BBT}, \cite{Lusky2006}, \cite{SW1971} and the references therein). The space $H_v^{0}$ coincides with the closure of the polynomials in $H_v^{\infty}$ and $H_v^{\infty}$ is canonically isomorphic to  the bidual $(H_v^{0})^{**}$ \cite{BS}.

The constant weight $v\equiv 1$ does not satisfy our assumptions. In this case, $H_v^{\infty}$ is the classical Banach space $H^\infty$ of bounded holomorphic functions on $\D$, and $H_v^{0}$ reduces to $\{ 0 \}$. We will state and prove some results for $H^\infty$, but when we refer to a weight in the statements, the constant weight is excluded, unless it is explicitly mentioned.  The so-called \textit{standard  weights} are defined by  $v_{\gamma}(r)=(1-r)^{\gamma}, $ $r>0,$ $\gamma>0,$ the spaces $H_{v_{\gamma}}^{\infty}$ and $H_{v_{\gamma}}^0$ are the classical \textit{Korenblum type growth Banach spaces} $A^{-\gamma}$ and $A^{-\gamma}_0,$ respectively. These spaces play an important role in interpolation and   sampling of holomorphic functions \cite[Chapters 4 and 5]{Korenblum}. Shields and Williams \cite{SW1971} proved that  $H_{v_{\gamma}}^{\infty}$ is isomorphic to $\ell^{\infty}$ and that  $H_{v_{\gamma}}^0$ is isomorphic to $c_0$; a result that was extended by Lusky \cite[Theorem 1.1]{Lusky2006}. A survey about weighted Banach spaces of analytic functions of the type considered in this paper and operators between them can be seen in  \cite{pepesurvey}.

A weight $v$ is called \textit{essential} if there is a constant $c>0$ such that, for each $z \in \D$,
$$c/v(z) \leq \sup\{ |f(z)| \ ; \|f\|_v \leq 1 \} \leq 1/v(z).$$
This terminology was introduced in \cite{BBT}, where this type of weights were investigated. Every standard weight $v_{\gamma}$ is essential. The following consequence of a deep result due to Abakumov and Doubtsov \cite{AD} (see also \cite[Theorem 2]{pepesurvey}) is useful later in the article.

\begin{lemma}\label{essential}
If the weight $v$ is essential, then there are $f_v \in H_v^\infty$ and $D(v)>0$ such that $1/v(z) \leq D(v) f_v(|z|)$ for each $z \in \D$. In particular $f_v(t) >0$ for each $t \in [0,1)$.
\end{lemma}
\begin{proof}
By \cite[Theorem 4]{AD} the weight $v$ is essential if and only if there is a sequence $(a_k)_{k=0}^{\infty}$ of non-negative numbers and there are constants $c_1, c_2 >0$ such that, for each $0 \leq r < 1$,
$$
c_1 \frac{1}{v(r)} \leq \sum_{k=0}^{\infty} a_k r^k \leq c_2 \frac{1}{v(r)}.
$$
Then the function $f_v(z):= \sum_{k=0}^{\infty} a_k z^k, \ z \in \D,$ belongs to $H(\D)$, and, for each $z \in \D$,
$$
v(z) |f_v(z)| \leq v(z) \sum_{k=0}^{\infty} a_k |z|^k \leq c_2.
$$
This implies $f_v \in H^\infty_v$. Moreover, $\frac{1}{v(r)} \leq (1/c_1) f_v(r)$ for each $0 \leq r < 1$, and the conclusion follows with $D(v)=1/c_1$.
\end{proof}

Given two functions  $f$  and $g$ defined in an (possibly unbounded) interval $I$ of real numbers, we write $f(x) \lesssim g(x)$ to denote the existence of $C>0$ such that $f(x) \leq C g(x)$ for each $x \in I$. In case $f(x) \lesssim g(x)  \lesssim f(x),$ we write $f(x)\cong g(x).$  If $a$ is an extreme point of the interval $I$ (possible infinity), we shall use Landau's notation $f(x)=O(g(x))$  as $x\to a$ and $f(x)=o(g(x))$ as $x\to a$.

We recall that a positive finite  Borel measure $\mu$ on $[0,1)$ is an {\it $s$-Carleson measure}, $s>0,$  if
$$ \mu([t,1))=O((1-t)^s), \text{ as $t\rightarrow 1^-$},$$
and a {\it vanishing $s$-Carleson measure} if
$$\mu([t,1))=o((1-t)^s) \text{ as $t\rightarrow 1^-$}.$$
A  $1$-Carleson measure and a vanishing $1$-Carleson measure will be simply called a Carleson measure and a vanishing Carleson measure, respectively. By \cite[Lemma 2]{GGM}, $\mu$ is a Carleson measure if and only if $\mu_n=O(\frac{1}{n})$ as $n\to \infty,$ and $\mu$ is a vanishing Carleson measure if and only if $\mu_n=o(\frac{1}{n})$ as $n\to \infty.$

Our notation for concepts from functional analysis and operator theory is standard. We refer the reader e.g.\ to  \cite{Conway} and \cite{meisevogt}. 	 Given a Banach space $X,$ we denote by ${X}^*$ its topological dual, that is, the space of  all continuous linear forms on $X,$ and  given a continuous operator $T$ on $X,$ we denote by ${T}^*$ its \emph{transpose}.

	\section{$I_{\mu}$ as an operator with values in $H(\D)$}

In this section $(X,\|\cdot\|)$ is a general Banach space of analytic functions, that is, a Banach space containing the polynomials such that the inclusion  $X\hookrightarrow (H(\D),\tau_0)$ is continuous. We analyze, for a general measure $\mu$,  when the operator $I_\mu:X\to H(\D)$ is well defined and continuous. We also  study the relation with the operator $H_\mu:X\to H(\D)$. 	
	
First, observe that, given $f(z)=\sum_{n=0}^{\infty}a_nz^n\in H(\D),$ we get $I_\mu(f)(0)=\int_{0}^{1}f(t)d\mu(t)$. So, if $I_{\mu}$ is well defined, then the restriction of $f$ to $[0,1)$ must belong to $L([0,1),\mu).$ In the next proposition we prove that this condition is also sufficient to get $I_{\mu}(f)\in H(\D)$.

\begin{proposition}\label{welldefinedL1}
	Let $\mu$ be a positive finite Borel measure on $[0,1)$. If $f\in H(\D)$ satisfies that its restriction to $[0,1)$ belongs to  $L_1([0,1),\mu)$, then $I_\mu(f)\in H(\D)$.
\end{proposition}

\begin{proof}
	Fix $f\in H(\D)\cap L_1([0,1),\mu)$. For $0\leq r<1,$ put $I^r_{\mu}(f)(z):=\int_{0}^{r}\frac{f(t)}{1-tz}d\mu(t)\in H(\D).$ Given $0<R<1,$ we get
	$$\sup_{z\in \D, |z|\leq R}|I_{\mu}(f)(z)-I_{\mu}^{r}(f)(z)|\leq  \frac{1}{1-R}\int_{r}^{1}|f(t)|d\mu(t).$$
	Hence, as $\int_{0}^{1}|f(t)|d\mu(t)<\infty$, we get $I_{\mu}^{r}(f)\rightarrow I_{\mu}(f)$  as $r\rightarrow 1$ in $\tau_0$. Hence,  $I_{\mu}(f)\in H(\D)$.
\end{proof}

If we look now at the Hilbert-type operator, we see that  $H_{\mu}(f)(0)=\sum_{n=0}^{\infty}\mu_na_n$. So, if $H_{\mu}$ is   well defined for  $f(z)=\sum_{n=0}^{\infty}a_nz^n\in H(\D)$, then  $\sum_{n=0}^{\infty}\mu_na_n$ must be convergent.  In the next proposition we prove that if we consider the stronger condition $\sum_{n=0}^{\infty}\mu_n|a_n|<\infty$, we even get that  $H_\mu(f)$ and  $I_\mu(f)$ are well defined and coincide.

\begin{proposition}\label{HI}
	Let $X\hookrightarrow (H(\D),\tau_0)$  be a Banach space of analytic functions and let $\mu$ be a positive finite Borel measure on $[0,1)$. If $\sum_{n=0}^{\infty}\mu_n|a_n|<\infty$ for each sequence $(a_n)_n$ such that  $f(z)=\sum_{n=0}^{\infty}a_nz^n\in X$, $z\in\D$, then $H_\mu:X\to H(\D)$ and $I_\mu:X\to H(\D)$ are well defined and $I_\mu=H_\mu$.
\end{proposition}

\begin{proof}
On the one hand, as  $(\mu_n)_{n\in \N_0}$ is decreasing, by hypothesis   $(\sum_{k=0}^{\infty}\mu_{n+k}a_k)_{n\in \N_0}$ is bounded, thus $H_{\mu}(f)(z)=\sum_{n=0}^{\infty}\left(\sum_{k=0}^{\infty}\mu_{n+k}a_k\right)z^n\in H(\D)$ for every $f\in X$, $f(z)=\sum_{n=0}^{\infty}a_nz^n$, $z\in\D$. On the other hand, the hypothesis yields that the restriction of  every $f\in X$ to $[0,1)$ belongs to $L_1([0,1),\mu)$. Actually, the monotone convergence theorem gives
	$\int_0^1 |f(t)| d\mu(t)\leq \sum_{n=0}^{\infty}\mu_n|a_n|,$ hence $I_\mu:X\to H(\D)$ is well defined by Proposition \ref{welldefinedL1}.
The coincidence of  $I_{\mu}$ and $H_{\mu}$ on $X$ follows by the dominated convergence theorem using the Taylor series of  $\displaystyle\frac{1}{1-tz}$.
	\end{proof}

We need Abel summability of a series to write $I_{\mu}$ in terms of the Hilbert-type operator below in Proposition \ref{BenDefinitHdGeneralX}.

\begin{definition}
	A series $\sum_{n=0}^{\infty}a_n$, $a_n\in \C,$   is said to be Abel summable to   $S\in \C$  if $\sum_{n=0}^{\infty}a_nr^n$ is convergent for every $0<r<1$ and    $\lim_{r\rightarrow 1^-}\sum_{n=0}^{\infty}a_nr^n=S$.  In this case we write  $S=(A)$-$\sum_{n=0}^{\infty}a_n$.
\end{definition}	
It is known that, if $\sum_{n=0}^{\infty}a_n=S$,  then $(A)-\sum_{n=0}^{\infty}a_n=S$.

\noindent

In order to state the following proposition we introduce some notation. For  $f\in H(\D)$ and $0<r\leq 1$ we set $f_r(z):=f(rz)$. It is clear that $\lim_{r\to 1^-}f_r=f$ in $(H(\D), \tau_0).$

\begin{proposition}\label{BenDefinitHdGeneralX}
Let $\mu$ be a positive finite Borel measure on $[0,1)$ and let $X\hookrightarrow (H(\D), \tau_0)$ be a Banach space of analytic functions. Assume that
 \begin{itemize}
\item[(a)] there is $M>0$ such that for each $f\in X$ there is $g\in L_1([0,1),\mu)$ satisfying $|f_r(t)|\leq g(t)$ for every $t\in [0,1)$, $0<r\leq 1$ and $\int_0^1 g(t)d\mu(t)\leq M\|f\|$.
\end{itemize}
Then $I_\mu:(X,\| \ \|) \to (H(\D),\tau_0)$ is a well defined continuous operator with representation
 $$I_{\mu}(f)(z)=\displaystyle\sum_{n=0}^{\infty}\left((A)-\sum_{k=0}^{\infty}\mu_{n+k}a_k\right)z^n, \ z\in \D,$$
for every $f(z)=\sum_{n=0}^{\infty}a_nz^n\in X.$
\noindent Whenever $\sum_{k=0}^{\infty}\mu_{n+k}a_k$ converges for every $n\in \N_0$, then  $I_{\mu}(f)=H_{\mu}(f)$ on $X.$
\end{proposition}

\begin{proof} By Proposition \ref{welldefinedL1} and the hypothesis on $X$, $I_\mu(f)\in H(\D)$ for every $f\in X$.
 We show that $I_\mu:X\to H(\D)$ is continuous. Let $(f_n)_n\subseteq X$ be a sequence convergent to $f\in X$.  Fix $0<r_0<1$, $\eps>0$  and let $n_0\in\N$ such that $\|f_n-f\|<\frac{\eps (1-r_0)}{M}$ for each $n\geq n_0$. From the hypothesis, for each $n\in  \N$ there is $g_n\in L_1([0,1),\mu)$ such that  $|f_n(t)-f(t)|\leq g_n(t)$ for each $t\in [0,1)$ and $\int_0^1 g_n(t)d\mu(t)\leq M\|f_n-f\|$. Hence, for $n\geq n_0$,
 \begin{eqnarray*}
	\sup_{|z|\leq r_0}|I_{\mu}(f_n)(z)-I_{\mu}(f)(z)|&\leq& \sup_{ |z|\leq r_0}\int_{0}^{1}\frac{|f_n(t)-f(t)|}{|1-tz|}d\mu(t)\\
	&\leq&\frac{1}{1-r_0}\int_{0}^{1}|f_n(t)-f(t)|d\mu(t)\leq \frac{1}{1-r_0} \int_{0}^{1}g_n(t)d\mu(t)\leq \\
	&\leq& \frac{1}{1-r_0}M\|f_n-f\|\leq \eps
\end{eqnarray*}

 \noindent Therefore, $I_\mu:(X,\| \ \|) \to (H(\D),\tau_0)$ is continuous.

Now, let us see that $I_{\mu}$ can be represented in terms of the Hilbert-type operator. First, we prove  that  given $f\in X$,   $\lim_{r\to 1^-}I_{\mu}(f_r)=I_{\mu}(f)$  in $\tau_0$. Consider $g\in L_1([0,1),\mu)$ satisfying (a) for $f$. Fix $0<R<1,$ $\eps>0$  and let $0<r_0<1$ be such that $\int_{r_0}^1g(t)d\mu(t)\leq\frac{(1-R)\eps}{4}$.  Since $\lim_{r\to 1^-}f_r=f$ in the compact open topology, there is $0<r_1<1$ such that $|f_r(z)-f(z)|\leq \frac{(1-R)\eps}{2}$ for $r\geq r_1$, $|z|\leq r_0$. Hence, for every $r\geq r_1$ we get
		 \begin{eqnarray}\label{eq_provabola}
	\sup_{ |z|\leq R}|I_{\mu}(f_r)(z)-I_{\mu}(f)(z)|&\leq& \sup_{ |z|\leq R}\int_{0}^{1}\frac{|f_r(t)-f(t)|}{|1-tz|}d\mu(t) \nonumber \\
	&\leq&\frac{1}{1-R}\int_{0}^{r_0}|f_r(t)-f(t)|d\mu(t)+\frac{2}{1-R}\int_{r_0}^{1}g(t)d\mu(t)	\nonumber \\
	&\leq&  \eps.
\end{eqnarray}

 Given $f(z)=\sum_{k=0}^{\infty}a_kz^k\in X$ and $0<r<1,$ then $f_r(z)=\sum_{k=0}^{\infty}a_kr^kz^k\in H^{\infty}.$  Observe that $I_{\mu}(f_r)=H_{\mu}(f_r)$.  Indeed, given  $z\in \D$ we get
	\begin{eqnarray*}
		I_{\mu}(f_r)(z)&=&\int_{0}^{1}\frac{f(rt)}{1-tz}d\mu(t)=\int_{0}^{1}\frac{\sum_{k=0}^{\infty}a_kr^kt^k}{1-tz}d\mu(t)=\sum_{k=0}^{\infty}a_kr^k\int_{0}^{1}\frac{t^k}{1-tz}d\mu(t)\nonumber\\
		&=&\sum_{k=0}^{\infty}a_kr^k\int_{0}^{1}t^k\left(\sum_{n=0}^{\infty}(tz)^n\right)d\mu(t)=\sum_{k=0}^{\infty}a_kr^k\left(\sum_{n=0}^{\infty}\mu_{n+k}z^n\right)=\sum_{k=0}^{\infty}\sum_{n=0}^{\infty}\mu_{n+k}a_kr^kz^n\nonumber\\
		&=&\sum_{n=0}^{\infty}\left(\sum_{k=0}^{\infty}\mu_{n+k}a_kr^k\right)z^n=H_{\mu}(f_r)(z).
	\end{eqnarray*}
	We can perform all the permutations above because, for a fixed $0<r<1$ and $0<R<1$ the series is absolutely convergent on $t\in[0,1]$ and  $|z|\leq R$.
	
	Since  $I_{\mu}(f_r)$ converges to $I_{\mu}(f)$ in $\tau_0$ as $r\rightarrow 1^-,$ we get $\lim_{r\rightarrow 1^-}\sum_{n=0}^{\infty}\left(\sum_{k=0}^{\infty}\mu_{n+k}a_kr^k\right)z^n=I_{\mu}(f)(z)$ in $\tau_0,$ and thus we get convergence on the coordinates, i.e.  $\lim_{r\rightarrow 1^-}\sum_{k=0}^{\infty}\mu_{n+k}a_kr^k=\frac{(I_{\mu}(f))^{(n)}(0)}{n!}$ for every $n\in \N.$ Therefore,  for every $f\in X,$
	$$I_\mu(f)(z)=\sum_{n=0}^{\infty}\left((A)-\sum_{k=0}^{\infty}\mu_{n+k}a_k\right)z^n, \ z\in \D.$$
\end{proof}

\begin{theorem}\label{HvH1}
Let $\mu$ be a positive finite Borel measure on $[0,1)$ and let $X\hookrightarrow (H(\D),\tau_0)$ be a Banach space of analytic functions. The operator $I_\mu:X\to H(\D)$ is well defined and continuous and admits the expression $I_\mu(f)=\sum_{n=0}^{\infty}\left((A)-\sum_{k=0}^{\infty}\mu_{n+k}a_k\right)z^n$ for $f(z)=\sum_{n=0}^{\infty}a_nz^n\in X$ and $I_\mu(f)=H_\mu(f)$ if $\sum_{k=0}^{\infty} \mu_{n+k}a_k$ is convergent for each $n\in\N_0$ in the following cases:

\begin{itemize}
\item[(i)] When $X=H_v^\infty$ and $\frac1v\in  L_1([0,1),\mu)$, $v$ being a  weight on $\D.$ If the weight $v$ is essential, then the condition $\frac1v\in  L_1([0,1),\mu)$ is also necessary in order to have $I_\mu:X\to H(\D)$ well defined.

\item[(ii)] When $X=H^1$ and $\mu$ is a Carleson measure.
\end{itemize}
\end{theorem}

\begin{proof}
 (i)  Just observe that  we can apply Proposition \ref{BenDefinitHdGeneralX} with $M=\int_{0}^{1}\frac{1}{v(t)} d\mu(t)$ and  $g=\frac{\|f\|_v}{v}\in  L_1([0,1),\mu)$, for a given $f\in H_v^\infty$. As   $\|f_r\|_v\leq \|f\|_v$ for each $f\in H_v^\infty$ and  $0<r\leq 1$, then $|f_r(t)|\leq \frac{||f_r||_v}{v(t)}\leq g(t)$ for $0\leq t<1$, $0<r\leq 1$.

Now assume that  $I_\mu:H_v^\infty\to H(\D)$ is well defined and that $v$ is essential. By Lemma \ref{essential} there are $f_v \in H_v^\infty$ and $D(v)>0$ such that $1/v(t) \leq D(v) f_v(t)$ for each $t \in [0,1)$. We have
$$
I_\mu(f_v)(0)=\int_0^1 f_v(t) d\mu(t) \geq \frac{1}{D(v)} \int_0^1\frac{1}{v(t)}d\mu(t),
$$
and $1/v(t) \in  L_1([0,1),\mu)$.

(ii) Let $C>0$ be such that $\mu_n\leq \frac{C}{n+1}$ for each $n\in\N_0$. Given $f(z)=\sum_{n=0}^{\infty}a_nz^n\in H^1$, consider $g(t)=\sum_{n=0}^{\infty} |a_n| t^n,$ $0\leq t<1$. The monotone convergence theorem combined with the Hardy inequality give $g\in L_1([0,1),\mu)$ and
$\int_0^1 g(t)d\mu(t)=\sum_{n=0}^{\infty}|a_n|\mu_n\leq C\sum_{n=0}^{\infty}\frac{|a_n|}{n+1}\leq C\pi\|f\|_{H^1},$
and then  (a) in Proposition \ref{BenDefinitHdGeneralX} is satisfied, since $|f_r(t)|\leq g(t)$   holds for $0\leq t<1$,  $0<r\leq 1$.
\end{proof}

\section{$I_\mu$ as an operator between weighted Banach spaces of analytic functions.}

Now  we study the  action of  $I_\mu$ on weighted Banach spaces of holomorphic functions defined by general weights as well as for $H^\infty$, and give a characterization of the boundedness and compactness of the operator in the case of standard weights. In what follows, $B_v^\infty$ denotes the closed unit ball of $H_v^\infty.$

\subsection{General weights}
\begin{lemma}
\label{techlem}
Let $\mu$ be a positive finite Borel measure on $[0,1)$ and let $v$ be weight on $\D$ such that $1/v\in L_1([0,1),\mu)$. Then $I_\mu:(B_v^\infty,\tau_0)\to H(\D)$ is continuous.
\end{lemma}
\begin{proof}
The hypothesis and Theorem \ref{HvH1}(i) imply that $I_\mu:H_v^\infty\to H(\D)$ is a well defined continuous linear operator. Let us see that for every $(f_n)_n\subseteq B_v^\infty$ convergent to $f\in  B_v^\infty$ in $\tau_0,$ then $(I_{\mu}(f_n))_n$ converges to $I_{\mu}(f)$ in $\tau_0$. If $(f_n)_n\subseteq B_v^\infty$, then $|f_n(t)|\leq \frac{1}{v(t)}$ for every   $0\leq t<1,$ and every $n\in \N.$ As  $\frac{1}{v(t)}\in L_{1}([0,1),\mu),$ the conclusion follows  proceeding analogously as in (\ref{eq_provabola}) in Proposition \ref{BenDefinitHdGeneralX}, taking $g=1/v$ and considering $f_n$ instead of $f_r.$
\end{proof}

In the next theorem we give sufficient conditions for the continuity of the operator $I_{\mu}$ between weighted Banach spaces  of holomorphic functions of type $H_{v}^\infty$. These conditions are stronger than the assumption  $\int_{0}^{1}\frac{1}{v(t)}d\mu(t)<\infty$, which appeared in  Theorem \ref{HvH1}.

\begin{theorem}
\label{Hv}
Let $\mu$ be a positive finite Borel measure on $[0,1)$ and let $v,w$ be weights on $\D.$  Consider the following statements:

\begin{itemize}

\item[(i)] $C(v,w):=\sup_{0\leq r<1}w(r)\int_{0}^{1}\frac{d\mu(t)}{v(t)(1-tr)}<\infty,$
		\item[(ii)]  $I_\mu: H_{v}^\infty\to H_{w}^\infty$ is well defined and continuous.
	\end{itemize}

Then (i) always implies (ii). If we assume in addition  that the weight $v$ is essential, then (i) and (ii) are equivalent.
	
\end{theorem}
\begin{proof}
If we assume (i) then $1/v\in L_1([0,1),\mu)$ (just take $r=0$). Thus, Theorem  \ref{HvH1}(i) implies that $I_\mu:H_v^\infty\to H(\D)$ is a well defined continuous operator. Moreover, for each $f\in H_v^\infty,$
$$\sup_{0\leq r<1} w(r)\max_{|z|=r}|I_\mu(f)(z)|\leq \|f\|_v \sup_{0\leq r<1}w(r)\int_{0}^{1}\frac{d\mu(t)}{v(t)(1-tr)},$$

hence $I_\mu(H_v^\infty)\subseteq H_w^\infty$ and (ii) is a consequence of the closed graph theorem.

For the converse, we assume (ii) and  that the weight $v$ is essential.  By Lemma \ref{essential} there are $f_v \in H_v^\infty$ and $D(v)>0$ such that $1/v(t) \leq D(v) f_v(t)$ for each $t \in [0,1)$. We have
	\begin{equation}\label{norm2}
		w(r)\int_{0}^{1}\frac{d\mu(t)}{v(t)(1-tr)}\leq  D(v) \sup_{z\in \D, |z|=r}w(z)\left|\int_0^1\frac{f_v(t)}{1-tz}d\mu(t)     \right|= D(v)\sup_{z\in \D, |z|=r}w(z)|I_{\mu}(f_v)(z)|.
	\end{equation}
	So,  if we assume $I_{\mu}(f_v)\in H_{w}^\infty,$ we get that the condition in (i) is also necessary.
\end{proof}

\begin{proposition}
\label{Hv0}
Let $\mu$ be a positive finite Borel measure on $[0,1)$ and let $v,w$ be weights on $\D$ such that   $1/v\in L_1([0,1),\mu)$.  The following conditions are equivalent:

\begin{itemize}
		\item[(i)]  $I_\mu: H_{v}^0\to H_{w}^0$ is continuous.
         \item[(ii)] $I_\mu:H_v^\infty\to H_w^\infty$ is  continuous and $I_\mu(1)\in H_w^0$.
\end{itemize}
	If the  equivalent conditions hold, then the bitranspose $I_{\mu}^{**}$ of $I_{\mu}: H_{v}^0\to H_{w}^0$ coincides with $I_{\mu}: H_{v}^\infty\to H_{w}^\infty,$ and $\|I_{\mu}\|_{\mathcal{L}(H_{v}^\infty, H_{v}^\infty)}=\|I_{\mu}\|_{\mathcal{L}(H_{v}^0, H_{v}^0)}.$
\end{proposition}

\begin{proof}
Assume that (i) holds. Then $I_\mu(1)\in H_w^0$  immediately. We can conclude by Lemma \ref{techlem}, proceeding  analogously as in \cite[Lemma 1]{BBJCesaro}.

If we assume that (ii) holds,  for each $0<r<1$,
\begin{equation}
\max_{|z|=r}|I_\mu(z^k)|=\sum_{n=0}^{\infty} \mu_{n+k}r^n\leq \sum_{n=0}^{\infty} \mu_{n}r^n=\max_{|z|=r}|I_\mu(1)|,
\end{equation}
\noindent we get $I_\mu(z^k)\in H_w^0$ for each $k\in\N$. Since the  polynomials $\mathcal{P}\subseteq H_v^0$ are norm dense in $H_v^0$, we get $I_\mu(H_v^0)\subseteq H_w^0$ and  $I_\mu: H_{v}^0\to H_{w}^0$ is continuous by the closed graph theorem.
 The last assertion follows reasoning analogously as in \cite[Lemma 1]{BBJCesaro}. 	
\end{proof}

\begin{proposition}\label{exemples C_mu(1)}
		Let $\mu$ be a positive finite Borel measure on $[0,1)$ and $v$ a weight on $\D$ such that $1/v\in L_1([0,1),\mu)$. Consider $I_{\mu}: H_v^{\infty}\rightarrow H(\D)$.   Then:
	\begin{itemize}
		\item[(i)] If  $\mu$ is a Carleson measure and $w$ is a weight such that  $w(r)|\log (1-r)|\rightarrow 0$ as $r\rightarrow 1^-$,  then $I_{\mu}(1)\in H_{w}^{0}.$ This condition is satisfied, for instance, in case $w$ is a standard weight.
		\item[(ii)] 	If $\mu$ is an $s$-Carleson measure for $0<s<1$, then $I_\mu(1)\in H_{v_\gamma}^{0}$ for each $\gamma>1-s$.
	\end{itemize}	
	
\end{proposition}

\begin{proof}
	By Theorem \ref{HvH1}(i), the hypothesis implies $I_{\mu}: H_v^{\infty}\rightarrow H(\D)$ is well defined and continuous. As $I_{\mu}(1)=C_{\mu}(1)$, where $C_{\mu}$ is the Cesàro-type operator associated with the Borel measure $\mu,$ \cite[Proposition 2]{BBJCesaro}  yields  the conclusions.
\end{proof}

\begin{proposition}\label{weaklycompact}
	Let $\mu$ be a positive finite Borel measure on $[0,1)$ and $v, w$ be weights on $\D$ such that  $1/v\in L_1([0,1),\mu)$  and $I_\mu: H_{v}^0\to H_{w}^0$ is continuous. Then
	$I_\mu: H_{v}^0\to H_{w}^0$ and	$I_\mu: H_{v}^{\infty}\to H_{w}^{\infty}$  are  compact if and only if $I_\mu(H_v^\infty)\subseteq H_w^0.$
\end{proposition}	

\begin{proof}
	By  Proposition \ref{Hv0},  $I_\mu: H_{v}^{\infty}\to H_{w}^{\infty}$ is continuous and the operators $I_{\mu}^{**}$ and $I_\mu$ coincide on $H_{v}^\infty$. Therefore $I_\mu(H_v^\infty)\subseteq H_w^0$ if and only if $I^{**}_{\mu}((H_v^0)^{**})\subseteq H_w^0$. This fact holds if and only if $I_\mu: H_{v}^0\to H_{w}^0$ is weakly compact by
	\cite[17.2.7]{Jarchow}. Now \cite[Corollary 2.5]{Alex} implies that this is equivalent to the compactness of $I_\mu: H_{v}^0\to H_{w}^0$. Finally, Proposition \ref{Hv0} together with Banach-Schauder theorem  yield that $I_\mu: H_{v}^{\infty}\to H_{w}^{\infty}$ is compact if and only if $I_\mu: H_{v}^0\to H_{w}^0$ is compact.
\end{proof}

\begin{theorem}\label{compactHv}
Let $\mu$ be a positive finite Borel measure on $[0,1)$ and let $v,w$ be weights on $\D.$ Assume that $1/v\in L_1([0,1),\mu)$.  Consider the following statements:
\begin{itemize}
\item[(i)] $\lim_{r\to 1^-}w(r)\int_{0}^{1}\frac{d\mu(t)}{v(t)(1-tr)}=0,$
		\item[(ii)]  $I_\mu: H_{v}^\infty\to H_{w}^0$ is  compact.
				\item[(iii)]  $I_\mu: H_{v}^\infty\to H_{w}^\infty$ is  compact and $I_{\mu}(1)\in  H_{w}^0.$
		\item[(iv)] $I_\mu: H_{v}^\infty\to H_{w}^0$ is   continuous.
	\end{itemize}

 Then (i) implies (ii), which imples (iii), and (iii) implies (iv). If  in addition  $v$ is essential, then the four conditions are equivalent.

\end{theorem}
\begin{proof}
Observe that the hypothesis (i) implies
that $I_\mu:H_v^\infty\to H_w^\infty$ is continuous by Theorem \ref{Hv}. Moreover, for each $f\in B_v^\infty$ and $0<r<1,$
$$w(r)	\max_{|z|=r}|I_\mu(f)(z)|\leq w(r)\int_{0}^{1}\frac{d\mu(t)}{v(t)(1-tr)}.	$$
Hence, $I_\mu(H_v^\infty)\subseteq H_w^0$. The compactness of $I_\mu$ follows by Proposition \ref{weaklycompact}, so we have that (i) implies (ii).   (ii) implies (iii) is trivial and (iii) implies (iv)  by Propositions  \ref{Hv0} and \ref{weaklycompact}.

Now, let  us assume  the weight $v$ is essential. By Lemma \ref{essential} there are $f_v \in H_v^\infty$ and $D(v)>0$ such that $1/v(t) \leq D(v) f_v(t)$ for each $t \in [0,1)$.  If condition  (iv) hold, then (i) follows immediately from $\lim_{r\to 1^-}w(r)|I_{\mu}(f_v)(r)|=0.$
\end{proof}

\begin{corollary}
\label{compact1}
Let $\mu$ be a positive finite Borel measure on $[0,1)$ and let $v,w$ be weights on $\D.$ Assume that $1/v\in L_1([0,1),\mu)$ and  that  $w(r)\leq 1-r$ for every $0<r<1$. Then $I_\mu: H_v^\infty\to H_w^0$ is compact.
\end{corollary}

\begin{proof}
If we define $g_r(t):=\frac{w(r)}{v(t)(1-tr)}$, $t\in[0,1)$, $0<r<1$, we have $|g_r(t)|\leq \frac{1}{v(t)}$ for every $0<r<1$, $0\leq t< 1$ and $\lim_{r\to 1^-}g_r(t)=0$ for all   $0\leq t< 1$. The dominated convergence theorem permits us to conclude that condition (i) in Theorem \ref{compactHv} holds.
\end{proof}

\subsection{Standard weights}

In this section we consider weighted Banach spaces of holomorphic functions determined by  standard weights $\vg(r)=(1-r)^{\gamma},$ $0\leq r<1,\ \gamma>0.$  In Theorems \ref{TeoremaContinuitatGammamenor1} and \ref{TeoremaCompacitatGammamenor1} we get that, for $\gamma>0$ and $\delta\in (-\gamma,1-\gamma),$  the operator $I_{\mu}:H_{\vg}^\infty\rightarrow H_{\vgd}^\infty$ is continuous  (compact) if and only if the Cesàro-type operator $C_{\mu}:H_{\vg}^\infty\rightarrow H_{\vgd}^\infty$ is so (see \cite[Theorems 10 and 11]{BBJCesaro}). First, we need two lemmata.

\begin{lemma}
\label{teclem1}
Let $\mu$ be a positive finite Borel measure on $[0,1)$, $\gamma>0,$ and $\delta\in (-\gamma,1-\gamma).$
The following assertions are equivalent:
\begin{itemize}
\item[(i)] $\sup_{0\leq r<1}(1-r)^{\gamma+\delta}\int_0^1\frac{d\mu(t)}{(1-t)^{\gamma}(1-tr)}<\infty$
\item[(ii)] $\sup_{0\leq r<1}(1-r)^{\gamma+\delta} \int_{0}^{1}\frac{d\mu(t)}{(1-tr)^{\gamma+1}}<\infty$.
\item[(iii)]   $\mu$ is a $(1-\delta)$-Carleson measure.
\end{itemize}
\end{lemma}
\begin{proof}
(ii)  is equivalent to (iii) by \cite[Theorem 10]{BBJCesaro}.  Since $\frac{1}{1-tr}\leq \frac{1}{1-t}$ for every $r,t\in [0,1)$, we get immediately that (i) implies (ii). We show that (iii) implies (i).  By hypothesis, there exists $C>0$ such that $\mu[t,1)\leq C(1-t)^{1-\delta}.$ Thus, using integration by parts and \cite[Lemma 2.5]{Dai}, we get
\begin{eqnarray*}
	\int_0^1\frac{d\mu(t)}{(1-t)^{\gamma}(1-tr)}&=&\mu([0,1))+\int_0^1 \left(\frac{1}{(1-t)^{\gamma}(1-tr)}\right)'\mu([t,1))dt\\
	&\leq& \mu([0,1))+C\int_0^1 \frac{\gamma(1-tr)+r(1-t)}{(1-tr)^{2}(1-t)^{\gamma+\delta}}dt\\ 	
	&\leq& \mu([0,1))+C(\gamma+r)\int_0^1 \frac{dt}{(1-tr)(1-t)^{\gamma+\delta}}\\ 	
	&\leq& \mu([0,1))+C'(\gamma+r)(1-r)^{-(\gamma +\delta)}.
	\end{eqnarray*}
\end{proof}

The proof of the next Lemma is completely analogous of that of Lemma \ref{teclem1}, using \cite[Theorem 11]{BBJCesaro}.
\begin{lemma}
\label{teclem2}
Let $\mu$ be a positive finite Borel measure on $[0,1)$, $\gamma>0,$ and $\delta\in (-\gamma,1-\gamma).$
The following assertions are equivalent:
\begin{itemize}
\item[(i)] $\lim_{r\to 1^-}(1-r)^{\gamma+\delta}\int_0^1\frac{d\mu(t)}{(1-t)^{\gamma}(1-tr)}=0.$
\item[(ii)] $\lim_{r\to 1^-}(1-r)^{\gamma+\delta} \int_{0}^{1}\frac{d\mu(t)}{(1-tr)^{\gamma+1}}=0$.
\item[(iii)]   $\mu$ is a vanishing $(1-\delta)$-Carleson measure.
\end{itemize}
\end{lemma}

\begin{theorem}\label{TeoremaContinuitatGammamenor1}
	Let $\mu$ be a positive finite Borel measure on $[0,1)$, $\gamma>0,$ and $\delta\in (-\gamma,1-\gamma).$ The following are equivalent:
	\begin{itemize}
		\item[(i)]  $I_{\mu}:H_{\vg}^\infty\rightarrow H_{\vgd}^\infty$ is  continuous.
		
		\item[(ii)]  $I_{\mu}:H_{\vg}^0\rightarrow H_{\vgd}^0$ is  continuous and $1/\vg\in L_1([0,1),\mu)$.

		\item[(iii)]   $\mu$ is a $(1-\delta)$-Carleson measure.
		
		\item[(iv)]  $\mu_n=O(\frac{1}{n^{1-\delta}})$.
			
	\end{itemize}
\end{theorem}

\begin{proof}
The equivalence of (i) and (iii) follows from Theorem \ref{Hv} and Lemma \ref{teclem1}. The equivalence between (iii) and (iv) is stated in \cite[Proposition 1]{CGP} without proof and a proof can be found in \cite[Corollary 4.6]{Blasco}. To see that these three equivalent conditions are equivalent to (ii), by Proposition \ref{Hv0} we only need to show  that $I_\mu(1)\in H_{\vgd}^0$ when (i) holds. By the equivalence between conditions (i) and (iii)  and by Theorem \ref{HvH1}(i),   $I_\mu(1)\in H_{\vgd}^0$ is satisfied for every $\gamma>0$   by Proposition \ref{exemples C_mu(1)}(i) when $-\gamma <\delta\leq 0$ and   by Proposition \ref{exemples C_mu(1)}(ii) when $\delta\in (0,1-\gamma)$ (just take $s=1-\delta$ and consider $H_{\vgd}^0$ instead of $H_{\vg}^0$ in the proposition).
\end{proof}

\begin{theorem}\label{TeoremaCompacitatGammamenor1}
	Let $\mu$ be a positive finite Borel measure on $[0,1)$, $\gamma>0,$ and $\delta\in (-\gamma,1-\gamma).$ The following are equivalent:
	\begin{itemize}
	
		\item[(i)]  $I_{\mu}:H_{\vg}^\infty\rightarrow H_{\vgd}^\infty$ is  compact.
			\item[(ii)]   $I_{\mu}:H_{\vg}^\infty\rightarrow H_{\vgd}^0$ is  compact.
		\item[(iii)]  $I_{\mu}:H_{\vg}^0\rightarrow H_{\vgd}^0$ is  compact and  $1/\vg\in L_1([0,1),\mu)$.
				\item[(iv)] $\mu$ is a  vanishing $(1-\delta)$-Carleson measure.
		\item[(v)]     $\mu_n=o(\frac{1}{n^{1-\delta}})$.

	\end{itemize}
	\end{theorem}

\begin{proof}
 (i) is equivalent to (ii) by Theorem  \ref{TeoremaContinuitatGammamenor1} and Proposition \ref{weaklycompact}. (ii) and (iii) are equivalent by Theorem \ref{HvH1}(i) and Proposition \ref{weaklycompact}. The equivalence with (iv) follows by Theorem \ref{compactHv} and Lemma \ref{teclem2}, and (iv)  is equivalent to (v)  by \cite[Theorem 11]{BBJCesaro}.
\end{proof}

In the case  $\delta\geq 1-\gamma$, $\gamma>0,$ we prove that  the continuity and compactness of the operator    $I_{\mu}:H_{\vg}^\infty\rightarrow H_{\vgd}^\infty$
are equivalent  to the well definition of $I_{\mu}:H_{\vg}^\infty\rightarrow H(\D)$.

\begin{theorem}\label{TeoremaContinuitatGammamajor1}
	Let $\mu$ be a positive finite Borel measure on $[0,1)$, $\gamma>0$.  The following are equivalent:
	\begin{itemize}
		\item[(i)]   $I_{\mu}:H_{\vg}^\infty\rightarrow H(\D)$ is  well defined (even continuous).
				\item[(ii)] $\int_{0}^{1}\frac{d\mu(t)}{(1-t)^{\gamma}}<\infty.$
		\item[(iii)] $\sum_{n=0}^{\infty}\mu_nn^{\gamma-1}<\infty.$
				\item[(iv)]  $I_{\mu}:H_{\vg}^\infty\rightarrow H_{v_1}^\infty$ is a compact operator, i.e.  $I_{\mu}:H_{\vg}^\infty\rightarrow H_{\vgd}^\infty$ is so  for every $\gamma+\delta \geq 1$.	
	 	\item[(v)]  $I_{\mu}:H_{\vg}^\infty\rightarrow H_{v_1}^\infty$ is well defined, i.e.   $I_{\mu}:H_{\vg}^\infty\rightarrow H_{\vgd}^\infty$ is so  for every $\gamma+\delta \geq 1$.	
				\end{itemize}
				In particular,  $I_{\mu}:H_{v_1}^\infty\rightarrow H_{v_{1}}^\infty$ is well defined if  and only if it is compact if and only if $\sum_n\mu_n<\infty.$
 If (i)-(v) are satisfied, then $I_{\mu}:H_{\vg}^0\rightarrow H_{v_1}^0$ is a compact operator, i.e.  $I_{\mu}:H_{\vg}^0\rightarrow H_{\vgd}^0$ is so  for every $\gamma+\delta \geq 1$, and $I_{\mu}(H_{\vg}^\infty)\subseteq H_{v_1}^0$ for every $\gamma>0.$
\end{theorem}
\begin{proof}
The equivalence between (i) and (ii) is a consequence of Theorem \ref{HvH1}(i), since $\frac{1}{v_\gamma(t)}=\frac{1}{(1-t)^\gamma}$.
 (ii) is equivalent to (iii) by the fact that $f(z)=\frac{1}{(1-z)^{\gamma}},\ z\in \D,$  satisfies $f(z)=\sum_{n=0}^{\infty}a_n(\gamma)z^n,\ z\in \D,$ with $a_n(\gamma) >0,$ $a_n(\gamma)  \cong n^{\gamma-1}$  (eg.\cite[p.12]{GGM}).
Corollary \ref{compact1} ensures that (i) implies $I_{\mu}:H_{\vg}^\infty\rightarrow H_{v_1}^\infty$ is compact.  Since $H_{v_{\gamma_1}}^\infty\hookrightarrow H_{v_{\gamma_2}}^\infty$ continuously for every $\gamma_2>\gamma_1>0,$  (iv) holds.
(iv) implies (v) and (v)  implies (i) are trivial. 	 The assertion about  $I_{\mu}:H_{\vg}^0\rightarrow H_{v_1}^0$ follows from the fact that if (iii) is satisfied, then $\mu$ is a $\gamma$-Carleson measure, thus Proposition \ref{exemples C_mu(1)} yields $I_{\mu}(1)\in H_{v_1}^0$ and Propositions \ref{Hv0} and  \ref{weaklycompact} yield the conclusion.
 \end{proof}

We get the following extension of \cite[Theorem 2.1 (ii) and (iii)]{AMS}, where the continuity of the operators is obtained for $\mu$ being the Lebesgue measure.

\begin{corollary}
\label{gammagamma}
	If $\mu$ is a Carleson but  not vanishing Carleson measure on $[0,1),$  then  $I_\mu:H_{\vg}^\infty\rightarrow H_{\vg}^\infty$  is continuous if and only if  $0<\gamma< 1$. In this case,   $I_\mu:H_{\vg}^\infty\rightarrow H_{\vg}^\infty$ is never compact.
	\end{corollary}

	\begin{proof}
 $\sum_{n=0}^\infty \mu_n<\infty$ implies $\lim_nn\mu_n= 0$ \cite[Theorem 3.3.1]{Knopp},  hence if  $\mu$ is not vanishing Carleson, then $\sum_{n=0}^\infty  \mu_n$ must diverge. Hence, if $\gamma\geq 1,$ then $I_\mu:H_{\vg}^\infty\rightarrow H_{\vg}^\infty$  is not well defined by the equivalence between (iii) and (v) in Theorem \ref{TeoremaContinuitatGammamajor1}. If $0<\gamma<1$  then $I_\mu:H_{\vg}^\infty\rightarrow H_{\vg}^\infty$ is  continuous but not compact by Theorems \ref{TeoremaContinuitatGammamenor1} and \ref{TeoremaCompacitatGammamenor1}.
		\end{proof}
	
 For $0<\gamma<1$, both the continuity and compactness of $I_\mu:H_{v_\gamma}^\infty\to H_{v_\gamma}^\infty$ are equivalent to those of the Ces\`aro type operator $C_\mu:H_{v_\gamma}^\infty\to H_{v_\gamma}^\infty$ by \cite[Theorems 10 and   11]{BBJCesaro} and Theorems \ref{TeoremaContinuitatGammamenor1} and   \ref{TeoremaCompacitatGammamenor1} above.  The next example shows a difference between the behaviour of  $I_{\mu}$ and $H_{\mu}$, and the behaviour of   $C_{\mu}$ when   $\gamma\geq 1$.

\begin{example}
	\label{expel}
For the positive finite Borel measure $\mu(t)=\frac{dt}{\log \frac{e}{1-t}}$, $t\in [0,1),$  we get:
\begin{itemize}
	\item[(i)] The operator  $I_{\mu}:H_{v_1}^\infty\rightarrow H(\D)$ is  not well defined.
	\item[(ii)] The operator  $H_{\mu}:H_{v_1}^\infty\rightarrow H(\D)$ is  not well defined.
	\item[(iii)] The Cesàro-type operator  $C_{\mu}:H_{v_1}^\infty\rightarrow H_{v_1}^\infty$ is  compact.
\end{itemize}
\end{example}

\begin{proof}  By \cite[Lemma 2.1 (vi)]{ExPelaez}, we get that  
$\mu$ is a positive finite Borel measure with moments $\mu_n\cong \frac{1}{n\log n},$ $n\in \N_0$. So, it is a vanishing Carleson measure but $\sum_{n=0}^{\infty}\mu_n$ diverges.  (i) is a consecuence of Theorem \ref{TeoremaContinuitatGammamajor1}. (ii) follows from the fact that $f(z)=\frac{1}{1-z}\in  H_{v_1}^\infty$ but $H_{\mu}(f)(0)=\sum_{n=0}^{\infty}\mu_n$. (iii) is immediate from \cite[Theorem 11]{BBJCesaro}.
\end{proof}

Recall that if $H_{\mu}$ is well defined on $H_{\vg}^{\infty}$, then $\sum_{n=0}^{\infty}\mu_na_n$ converges in $\C$ for each $f(z)=\sum_{n=0}^{\infty}a_nz^n\in H_{\vg}^{\infty}$. This is satisfied if $\sum_{n=0}^{\infty}\mu_n|a_n|<\infty$ for each $f(z)=\sum_{n=0}^{\infty}a_nz^n\in H_{\vg}^{\infty}$.
The   following result gives us information about the convergence of the last series for $s$-Carleson measures.

\begin{proposition}{\cite[Theorem 18]{Pavlovic2014}, \cite[Theorem 2.1]{Jevtic2017}}\label{Thm_coef_sumable}
\label{PJ}
Let $\gamma>0$ and  $s\geq 1.$  Then $\sum_{n=0}^{\infty}\frac{|a_n|}{(n+1)^{s}}<\infty$  for every $\sum_{n=0}^{\infty}a_nz^n\in H_{\vg}^{\infty}$  if and only if  $s>\gamma+1/2.$  In particular:
\begin{itemize}
	\item[(i)]  If $0<\gamma<1/2,$ then    $\sum_{n=0}^{\infty}\frac{|a_n|}{n+1}<\infty$ for every $\sum_{n=0}^{\infty}a_nz^n\in H_{\vg}^{\infty}$.
	\item[(ii)] If $\gamma\geq 1/2,$ then there exists   $\sum_{n=0}^{\infty}a_nz^n\in H_{\vg}^{\infty}$  such that  $\sum_{n=0}^{\infty}\frac{|a_n|}{n+1}$  diverges.
	\end{itemize}
\end{proposition}

As a consequence of Proposition \ref{Thm_coef_sumable}(ii),  $H_{\vg}^{\infty}\nsubseteq H^1$ for $\gamma\geq 1/2,$ as a consequence of   the Hardy inequality.

\begin{corollary}\label{cont_H}
	Let $\mu$ be a positive finite Borel measure on $[0,1)$,  		$\gamma>0,$ and $\int_{0}^{1}\frac{d\mu(t)}{v_{\gamma}(t)}<\infty.$ Then  $H_{\mu}: H_{\vg}^{\infty}\rightarrow H(\D)$ is well defined and continuous and	 $H_{\mu}=I_{\mu}$ on $H_{\vg}^{\infty} $ in the following cases:
	
		\begin{itemize}
			\item[(a)] 		If $\gamma <\frac{1}{2}$ and $\mu$ is a Carleson measure. In this case, $H_{\mu}: H_{\vg}^{\infty}\rightarrow H_{\vgd}^{\infty}$ and $H_{\mu}: H_{\vg}^0\rightarrow H_{\vgd}^0$ are continuous for  $\delta=0$ and compact  for $\delta>0$.
	                 \item[(b)] 		  If $\gamma\geq \frac{1}{2}$ and $\mu$ is an  $s$-Carleson measure for $s>\gamma+1/2$.  In this case, $H_{\mu}: H_{\vg}^{\infty}\rightarrow H_{\vgd}^{\infty}$ and $H_{\mu}: H_{\vg}^0\rightarrow H_{\vgd}^0$ are continuous for $\delta = 1-s$ and compact for  $\delta >1-s$.
				\end{itemize}

\end{corollary}
\begin{proof}
This is a consequence of Proposition \ref{HI} and Proposition \ref{PJ}, together with Theorems \ref{TeoremaContinuitatGammamenor1}, \ref{TeoremaCompacitatGammamenor1} and \ref{TeoremaContinuitatGammamajor1}.
\end{proof}

\section{The generalized Hilbert operator $H_{\mu}$ on solid hulls and cores}

The solid hull  and the solid core of $H_{\vg}^{\infty}$ are also   of interest in the study of the well definition of $H_{\mu}$ on this space. In this context we identify an analytic function $f(z)=\sum_{n=0}^{\infty}a_nz^n\in H(\D)$ with the sequence of its Taylor coefficients $(a_n)_n$. Given  $A$  a vector space of complex sequences containing  the space of all the sequences with finitely many non-zero coordinates, the \emph{solid hull} of $A$ is
$$S(A):=\{(c_n)_n: \ \exists(a_n)_n\in A \text{ such that } |c_n|\leq |a_n| \ \forall n\in \N\}\supseteq A,$$
and the \emph{solid core} of $A$ is
$$s(A):=\{(c_n)_n: \ (c_na_n)_n\in A \ \forall (a_n)_n\in \ell_{\infty}\}\subseteq A.$$

In the next result we describe the solid hull  and the solid core of the spaces under consideration.

	\begin{proposition}{\cite[Theorems 8.2.1 and  8.3.4]{LlibreVukotic}}\label{Solid}
		For every $\gamma>0$,
		$$S(H_{v_{\gamma}}^{\infty})=\left\{(a_m)_m:\ \sup_{n\in \N_0}\left(\sum_{m=2^n}^{2^{n+1}-1}\frac{|a_m|^2}{(m+1)^{2\gamma}}\right)^{1/2}<\infty\right\},$$
	and
		$$s(H_{v_{\gamma}}^{\infty})=\left\{(a_m)_m:\ \sup_{n\in \N_0}\left(\sum_{m=2^n}^{2^{n+1}-1}\frac{|a_m|}{(m+1)^{\gamma}}\right)<\infty\right\}$$
	Both are Banach spaces of analytic functions for the norms $\|f\|_{S(H_{\vg}^{\infty})}$ and $\|f\|_{s(H_{\vg}^{\infty})}$ defined by the supremum.
\end{proposition}

\begin{proposition} Let $\gamma>0$ and let $\mu$ be an $s$-Carleson measure,  for $s>0$.

\begin{itemize}
\item[(i)] 	If   $s>\gamma+1/2$, then $H_{\mu}:S(H_{\vg}^{\infty})\rightarrow H(\D)$ is a well defined continuous operator and $H_\mu=I_\mu$ on $S(H_{\vg}^{\infty})$. If in addition $\gamma\geq \frac12,$ then $H_{\mu}:S(H_{\vg}^{\infty})\rightarrow S(H_{\vg}^{\infty})$ is a well defined continuous operator.
\item[(ii)]  If $s>\gamma$, then
  $H_{\mu}:s(H_{\vg}^{\infty})\rightarrow H(\D)$ is a well defined continuous  operator and  $H_{\mu}=I_{\mu}$ on $s(H_{\vg}^{\infty})$.   If in addition $\gamma\geq 1,$  then $H_{\mu}:s(H_{\vg}^{\infty})\rightarrow s(H_{\vg}^{\infty})$ is a well defined continuous  operator.
	\end{itemize}

\end{proposition}

\begin{proof} (i)   Given $f(z)=\sum_{m=0}^{\infty}a_mz^m\in   S(H_{v_{\gamma}}),$  by Proposition \ref{Solid}  there exists $C>0$ such that $\displaystyle\left(\sum_{m=2^n}^{2^{n+1}-1}\frac{|a_m|^2}{(m+1)^{2\gamma}}\right)^{1/2}<C \text{ for every } n\in\N_0.$ As $\|\ \|_1\leq d^{1/2}\|\ \|_2$   in $\C^d,$ with $d=2^n$, we get
	$$\displaystyle\sum_{m=2^n}^{2^{n+1}-1}\frac{|a_m|}{(m+1)^{\gamma}}\leq 2^{n/2}C.$$
	Therefore, for every $\delta>1/2$,
	\begin{equation}
	\label{hullestimate}
	\sum_{n=0}^{\infty}\sum_{m=2^n}^{2^{n+1}-1} \frac{|a_m|}{(m+1)^{\gamma+\delta}}\leq
\sum_{n=0}^{\infty}  \frac{1}{2^{n \delta}} \sum_{m=2^n}^{2^{n+1}-1} \frac{|a_m|}{(m+1)^{\gamma}} \leq
C\sum_{n=0}^{\infty}\frac{1}{2^{n(\delta-1/2)}}<\infty.
	\end{equation}
	
	\noindent Hence, if $\mu$ is an $s$-Carleson measure for $s>\gamma+1/2,$ then  $\sum_{n=0}^{\infty}\mu_{n}|a_n|<\infty$. We conclude by Proposition \ref{HI} and Proposition \ref{BenDefinitHdGeneralX}  that $H_{\mu}:S(H_{\vg}^{\infty})\to H(\D)$ is a well defined continuous operator and $H_\mu=I_\mu$ on $S(H_{\vg}^{\infty})$. Indeed, to apply Proposition \ref{BenDefinitHdGeneralX} take
	$g(t)=\sum_{n=0}^{\infty} |a_n|t^n$ for $f(z)=\sum_{n=0}^{\infty} a_nz^n\in S(H_{\vg}^{\infty})$. \eqref{hullestimate} gives, for  $M>0$ such that $\mu_n\leq \frac{M}{(n+1)^s}$, $n\in \N_0,$
	
	$$\int_0^1g(t)d\mu(t)=\sum_{n=0}^{\infty} \mu_n |a_n| \leq M\sum_{n=0}^{\infty}\frac{1}{2^{n(s-\gamma-1/2)}} \|f\|_{S(H_{\vg}^{\infty})}.$$
	
In this case,	let us see   that  $H_{\mu}(S(H_{\vg}^{\infty}))\subseteq S(H_{\vg}^{\infty})$ if $\gamma\geq 1/2.$ Let $M=\sum_{n=0}^{\infty}\mu_n|a_n|$ for $f(z)=\sum_{n=0}^{\infty} a_nz^n\in S(H_{\vg}^{\infty})$. For every $n\in \N_0$ we get
	$$\sum_{m=2^n}^{2^{n+1}-1}\frac{|\sum_{k=0}^{\infty}\mu_{k+m}a_k|^2}{(m+1)^{2\gamma}}\leq   2^n\frac{M^2}{(2^n+1)^{2\gamma}}\leq M^2 2^{n(1-2\gamma)}\leq M^2.$$
	The continuity holds by  the closed graph theorem, since $S(H_{\vg}^{\infty})\hookrightarrow H(\D)$ continuously.	
	
	\noindent
(ii) Removing the factor $2^{n/2}$ in the proof of (i),  we get that 	if $\mu$ is an $s$-Carleson measure for $s>\gamma,$ then there exists $M>0$ such that $\sum_{k=0}^{\infty}\mu_{n+k}|a_k|<M$ for every $n\in \N_0,$ thus,  $H_{\mu}(f)(z)=\sum_{n=0}^{\infty}\left(\sum_{k=0}^{\infty}\mu_{n+k}a_k\right)z^n \in H(\D)$ and  $H_{\mu}=I_{\mu}:s(H_{\vg}^{\infty})\to H(\D)$ is continuous. Let us see that  $H_{\mu}(s(H_{\vg}^{\infty}))\subseteq s(H_{\vg}^{\infty})$ if $\gamma\geq 1$.  Indeed, for  every $n\in \N_0,$
	$$\sum_{m=2^n}^{2^{n+1}-1}\frac{|\sum_{k=0}^{\infty}\mu_{k+m}a_k|}{(m+1)^{\gamma}}\leq   2^n\frac{M}{(2^n+1)^{\gamma}}\leq M 2^{n(1-\gamma)}\leq M.$$
The continuity holds again by the closed graph theorem.
\end{proof}

\section{Hilbert-type operators associated to summable moment measures}

In this section we deal with positive finite Borel measures $\mu$  on $[0,1)$ such that $\sum_{n=0}^{\infty} \mu_n<\infty$. This condition implies  $\mu$ is  a vanishing Carleson measure \cite[Theorem 3.3.1]{Knopp}.

 Proceeding as in the proof of \cite[Theorem 7]{BBJCesaro},  we obtain that this condition gives the equivalence of the boundedness and compactness of the Hilbert-type operator on  $H^\infty$.	 We include the proof for the sake of completeness.

 \begin{theorem}\label{sum_mu_n_implica_compact}
 	Let $\mu$ be a positive finite Borel measure on $[0,1)$ such that  $\int_{0}^{1}\frac{d\mu(t)}{1-t}<\infty$, i.e. $\sum_{n=0}^{\infty}\mu_n<\infty$.	Then  $I_\mu: H^\infty\rightarrow H^\infty$  is compact and $H_{\mu}=I_{\mu}.$	
 \end{theorem}

 \begin{proof}
 	Let us assume without loss of generality that $\mu([0,1))=1$. By   \cite[Theorem 1.2]{GirelaMerchan}, $I_{\mu}$ and $H_{\mu}$ are continuous on $H^\infty$ and $H_{\mu}=I_{\mu}.$ By the proof of \cite[Lemma 6]{BBJCesaro} we only need to show that $(\|I_\mu(f_n)\|_\infty)_n$ is convergent to 0 whenever $(f_n)_n\subseteq B_{H^\infty}$ is convergent to 0 in $\tau_0$. Let $\eps>0$. By  hypothesis, there is $s_0\in (0,1)$ such that
 	\begin{equation}
 		\label{p1}
 		\int_{s_0}^{1}\frac{d\mu(t)}{1-t}<\frac{\eps}{2}.
 	\end{equation}
 	We have that
 	\begin{equation}
 		\label{p2}
 		|I_\mu(f)(z)|=\left|\int_0^1\frac{f(t)}{1-tz}d\mu(t)\right|
 	\end{equation}
 	\noindent  for all $f\in H^\infty$, $z\in \D$. For $z\in \D$, $t\in [s_0,1]$ and $f\in B_{H^\infty}$ we have $|\frac{f(t)}{1-tz} |\leq \frac{1}{1-t}$. Hence, by \eqref{p1}, for each $n\in\N$ we have
 	
 	\begin{equation}
 		\label{p3}
 		\left|\int_{s_0}^1\frac{f_n(t)}{1-tz}d\mu(t)\right|<\frac{\eps}{2}.
 	\end{equation}
 	Take $n_0\in \N$ such that, for every $n\geq n_0,$  $|f_n(\omega)|<\frac{\eps(1-s_0)}{2}$ for every $|\omega|\leq s_0.$	 For $z\in \D$ and $n\geq n_0$ we also have
 	
 	\begin{equation}
 		\label{p4}
 		\left|\int_0^{s_0}\frac{f_n(t)}{1-tz}d\mu(t)\right|\leq\int_0^{s_0}\frac{\max_{\omega\in s_0\overline{\D}}|f_n(\omega)|}{1-s_0}d\mu(t)< \frac{\eps}{2}.
 	\end{equation}
 	Putting together \eqref{p2}, \eqref{p3} and \eqref{p4} we conclude that $\|H_\mu(f_n)\|_{\infty}<\eps$ for $n\geq n_0$, and so, $H_\mu: H^\infty\rightarrow H^\infty$,   is compact.	
 \end{proof}

 Theorem \ref{sum_mu_n_implica_compact} permits us to complement \cite[Theorem 1.2]{GirelaMerchan} with more equivalent conditions.

 \begin{theorem}\label{Cont_Hinfty}
 	Let $\mu$ be a positive finite Borel measure on $[0,1).$ The following conditions are equivalent:
 	\begin{itemize}
 		\item[(i)] $H_\mu: H^\infty \rightarrow H^\infty$ is continuous.
 		\item[(ii)] $H_\mu: H^\infty \rightarrow H^\infty$ is compact.
 		\item[(iii)] $\int_{0}^{1}\frac{d\mu(t)}{1-t}<\infty.$
 		\item[(iv)] $\sum_{n=0}^{\infty}\mu_n<\infty.$
		\item[(v)]$H_\mu: A(\D)\rightarrow A(\D)$ is continuous.
 		\item[(vi)] $H_\mu: A (\D)\rightarrow A(\D)$ is compact. 	
		\end{itemize}
 	If the assertions hold, then we get $H_{\mu}=I_{\mu}$ on $H^\infty$.
 \end{theorem}
 \begin{proof}
 The equivalence between (i), (iii) and (iv) is given in  \cite[Proof of Theorem 1.2]{GirelaMerchan}. Theorem \ref{sum_mu_n_implica_compact} yields that (iv) implies (ii), and, since (ii) implies (i) we have the equivalence of the first four statements. Since $A(\D)$ is a subspace of $H^\infty$, we get that (i) implies (v) if we show   that (iv) yields $H_\mu(A(\D))\subseteq A(\D)$. For each $k\in\N_0$, $H_\mu(z^k)=\sum_{n=0}^{\infty}\mu_{n+k}z^n\in A(\D)$. We conclude since the polynomials $\mathcal{P}$ are dense in $A(\D)$. Finally, (v) implies (iv) since $H_\mu(1)(1)=\sum_{n=0}^{\infty}\mu_n$. Trivially, (vi) implies (v), and (v) and (ii) imply (vi).
 \end{proof}

 In what follows, we  study the compactness and nuclearity of the Hilbert-type operator on different spaces of holomorphic functions on which the operator is well-defined under the assumption $\sum_{n=0}^{\infty}\mu_n<\infty.$ We recall that given two Banach spaces $X$ and $Y$, we say that an operator $T:X\rightarrow Y$ is \emph{nuclear} if there exist $(x_n^*)_n\subseteq X^*$ and $(y_n)_n\subseteq Y$ such that $\sum_{n=1}^{\infty} \|x_n^*\|\|y_n\|<\infty$ and $T=\sum_{n=1}^{\infty}x_n^*\otimes y_n= \sum_{n=1}^{\infty}x_n^*(\cdot) y_n.$ Nuclear operators are compact.

In the rest of this section, we consider the following general spaces of analytic functions:  
$$
\ell^p_A:= \{f(z)=\sum_{n=0}^{\infty}a_nz^n\in H(\D): \ \|f\|^p_p:=\sum_{n=0}^{\infty}|a_n|^p<\infty\} \text{ for $1 \leq p < \infty.$}
$$
$$
\ell^\infty_A:= \{f(z)=\sum_{n=0}^{\infty}a_nz^n\in H(\D): \ \|f\|_\infty:=\sup_n|a_n|<\infty\}. 
$$

These spaces have been thoroughly investigated in \cite{cheng}. The most relevant examples are the Wiener algebra (for $p=1$) $\ell^1_A=A(\T)$, and (for $p=2$) the classical  Hardy space $\ell^2_A=H^2$.  In the rest of the paper we identify a function $f(z)=\sum_{n=0}^{\infty}a_nz^n\in \ell^p_A,$ $1 \leq p \leq  \infty,$ with its coefficients $(a_n)_n$.

 \begin{proposition}
Let $\mu$ be a positive finite Borel measure on $[0,1)$ such that $\sum_{n=0}^{\infty} \mu_n<\infty$. Then $H_\mu:\ell^\infty_A\to \ell^\infty_A$ is well defined and $H_\mu=I_\mu$ on $\ell^\infty_A$.
\end{proposition}
 \begin{proof}
If $M=\sum_{n=0}^{\infty} \mu_n$, then it is immediate that $\|H_\mu((a_n)_n)\|_\infty\leq M\|(a_n)_n\|_\infty$ for all $(a_n)_n\in l_\infty$. $I_\mu: \ell^\infty_A\to H(\D)$ is well defined by Proposition \ref{BenDefinitHdGeneralX}. Indeed, for $f(z)=\sum_{n=0}^{\infty}a_n z^n$ with $(a_n)_{n}\in l_\infty$, we consider $g(z)=\sum_{n=0}^{\infty} |a_n|z^n $. Then $\int_{0}^{1} g(t)d\mu(t)\leq \|(a_n)\|_\infty \sum_{n=0}^{\infty} \mu_n$.
\end{proof}

	\begin{lemma}
		\label{nuclearcompact}
		Let $T:\ell^\infty_A\to \ell^\infty_A$, be given by $x\mapsto \sum_{k=0}^{\infty} \langle y_k,x\rangle e_k$, with $y_k\in \ell^1_A$ for each $k\in\N_0$ and $(\|y_k\|_1)_{k\in\N_0}$ bounded. Then, 	 for $1\leq p\leq  \infty$ and $p':=\frac{p}{p-1}$ if $1<p<\infty$, $p'=\infty$ if $p=1,$ and $p'=1$ if $p=\infty$, we have: 
		
		\begin{itemize}
			\item[(a)]  If $(\|y_k\|_{p'})_{k\in\N_0}\in l_1$ then $T(\ell^p_A)\subseteq \ell^1_A$, and the restriction $T: \ell^p_A\to \ell^1_A$ is a nuclear operator.
			\item[(b)]  If $(\|y_k\|_{p'})_{k\in\N_0}\in l_q,$ $q\neq \infty,$ then $T(\ell^p_A)\subseteq \ell^q_A$ and the restriction $T: \ell^p_A\to \ell^q_A$ is a compact operator.
		\end{itemize}
	\end{lemma}	
	\begin{proof}
 Observe that $(\ell^p_A)^*=\ell^{p'}_A$ if $1\leq p<\infty$ and $\ell^1_A\subseteq  (\ell^\infty_A)^*$ for the case $p=\infty$.
		Assertion (a) follows from the very definition of a nuclear operator. Under the assumptions of (b), it follows immediately that $T(\ell^p_A)\subseteq \ell^q_A$. Moreover, for the restriction $T: \ell^p_A\rightarrow \ell^q_A$ we have $T=\|\cdot\|-\lim_n\sum_{k=0}^{n} \langle y_k,x\rangle e_k$, which certainly implies the compactness of $T$. Actually, if  $(y_k)_k$ is a sequence  in $\ell^{p'}_A$ such that $(\|y_k\|_{p'})_k\in l_q$,  then for each $x\in B_{\ell^{p}_A}$ we get
		$$\left\|T(x)-\sum_{k=0}^{k_0}\langle y_k,x\rangle e_k\right\|_q\leq \left(\sum_{k=k_0+1}^{\infty}\|y_k\|_{p'}^q\right)^\frac{1}{q}.$$
		
	\end{proof}
	
	For a non increasing sequence $(\mu_n)_{n\in\N_0}$ of positive numbers such that $\sum_{n=0}^{\infty}\mu_n<\infty$, we consider the Hilbert and Ces\`aro type operators defined on $\ell^\infty_A$, i.e. $H_\mu: \ell^\infty_A\to \ell^\infty_A$, $H_\mu((a_n)_{n\in\N_0})=(\sum_{n=0}^{\infty}\mu_{n+k}a_n)_{k\in\N_0}$, $C_\mu: \ell^\infty_A\to \ell^\infty_A$, $C_\mu((a_n)_{n\in\N_0})=(\mu_{k}\sum_{n=0}^{k}a_n)_{k\in\N_0}$. These operators are well defined. We see that both satisfy the hypothesis of Lemma \ref{nuclearcompact}.  In fact, taking $y_k^h=(\mu_{n+k})_{n\in\N_0}$ for $H_\mu$ and $y^c_k=(a^k_j)_{j\in\N_0}$, where $a^k_j=\mu_k$ if $j\leq k$ and $a_j^k=0$ for $j>k$, we get
	
	$$\|y_k^h\|_1\leq \sum_{n=0}^{\infty} \mu_n \quad \text{for all } k\in\N_0$$
	
	\noindent and
	
	$$\|y_k^c\|_1\leq \sup\{(n+1)\mu_n: \ n\in\N_0 \}<\infty  \quad  \text{for all } k\in\N_0.$$
	
	\noindent For the last estimate we have used the fact that when $(\mu_n)_{n\in\N_0}$ is a decreasing sequence of positive numbers such that $\sum_{n=0}^{\infty} \mu_n<\infty,$ then $\lim_n n\mu_n=0$.
	
	\begin{proposition}\label{wienner}
Let $(\mu_n)_{n\in\N_0}$  be a non increasing sequence of positive numbers such that $\sum_{n=0}^{\infty}\mu_n<\infty$. The restrictions 		$H_\mu^1: \ell^1_A \to \ell^1_A$ and $C_\mu^1:\ell^1_A\to \ell^1_A$ of the Hilbert and Ces\`aro type operators are  nuclear operators. Moreover, the operator $H_\mu: \ell^\infty_A \to \ell^\infty_A$ satisfies $H_\mu=(H_\mu^1)^*$, and it is also a compact operator.
	\end{proposition}
	
	\begin{proof}
		The result for the restrictions follows from Lemma \ref{nuclearcompact} (a) for $p=1$, $p'=\infty$, and the equalities $\|y_k^h\|_\infty=\| y_k^c\|_\infty=\mu_{k},$ $k\in \N_0.$ We observe to conclude that $(H_\mu^1)^{*}=H_\mu$. In fact, $(H_{\mu}^1)^*$ is $w^*-w^*$ continuous, $\langle (H_\mu^1)^*(e_j),e_n\rangle=\langle e_j,H_\mu^1(e_n)\rangle=\langle y_j^h,e_n\rangle=\mu_{j+n}$ and $\langle H_\mu(e_j),e_n\rangle=\langle y_n^h,e_j\rangle=\mu_{j+n}$.
	\end{proof}

  \begin{proposition}
  	
		\label{lpl1}
		Let $\mu$ be a positive finite Borel measure on $[0,1)$ such that $\sum_{n=0}^{\infty}  \mu_n<\infty$  and  with  $1\leq p<\infty$  satisfying $(\|(\mu_{n+k})_n\|_{p})_k\in l_1$. The Hilbert-type operator $H_{\mu}:H^\infty\to A(\T)$ is a well defined compact linear operator.
	\end{proposition}
	\begin{proof}
		Since $H^\infty$ is a Grothendieck space and $A(\T)$ is separable and has the Schur property, we only need to show that $H_\mu$ is well defined. If $f(z)=\sum_{n=0}^{\infty}  a_n z^n\in H^\infty$, then $(a_n)_n\in l_q$ for each $q>1$. In particular, $(a_n)_n\in l_{p'}$. We have
		$H_\mu(f)=\sum_{k=0}^{\infty}(\sum_{n=0}^{\infty}\mu_{n+k}a_n) z^k$. For each $k\in \N_0$,
				$$\left|\sum_{n=0}^{\infty}\mu_{n+k}a_n \right|\leq \|(a_n\mu_{n+k})_n\|_1\leq \|(a_n)_n\|_{p'}\|(\mu_{n+k})_n\|_p,$$
				\noindent and we conclude from the hypothesis.
\end{proof}

	\begin{theorem}
		\label{compactw}
		Let $\mu$ be a positive finite Borel measure on $[0,1)$. The following are equivalent:
		
		\begin{itemize}
			\item[(i)] The operators $H_\mu: A(\T)\to A(\T)$ and $C_\mu:A(\T)\to A(\T)$ are well defined.
			\item[(ii)] The operators $H_\mu: A(\T)\to A(\T)$ and $C_\mu:A(\T)\to A(\T)$ are nuclear.
			\item[(iii)] $\sum_{n=0}^{\infty} \mu_n<\infty$.
		\end{itemize}
	\end{theorem}
	\begin{proof}
		We use the identification $A(\T)=\ell^1_A$. If the operators are well defined, then  $H_\mu(1)=C_\mu(1)=\sum_{n=0}^{\infty}\mu_nz^n$, and we have (i) implies (iii). Conversely, if we assume (iii), we get (ii)  from Proposition \ref{wienner}.
	\end{proof}
	
	\begin{proposition}
		\label{rcarleson}
		Let  $(\mu_n)_{n\in\N_0}$ be a non increasing sequence of positive numbers such that $\mu_n=O(\frac{1}{n^r})$ for some $r>1$. For each $1\leq q<\infty$ there is $p_0$ such that  $H_\mu(\ell^p_A)\subseteq \ell^q_A $ and $C_\mu(\ell^p_A)\subseteq \ell^q_A $ for each $1<p<p_0$. Moreover, the corresponding restrictions $H_\mu: \ell^p_A\to \ell^q_A$  and $C_\mu: \ell^p_A\to \ell^q_A$ are nuclear if $q=1$ and compact if $q>1$. For the particular case $p=q=2$ we get $H_\mu: \ell^2_A\to \ell^2_A$ and $C_\mu: \ell^2_A\to \ell^2_A$ are well defined compact operators, and they are nuclear if $r>\frac32$.
	\end{proposition}
	\begin{proof}
		Let $M>0$ such that $\mu_n\leq\frac{M}{(n+1)^r},$  $n\in \N_0$. For $H_\mu(x):=\sum_{k=0}^{\infty} \langle y_k^h,x\rangle e_k$, $y_k^h=(\mu_{n+k})_{n\in\N_0}$,  $k\in\N,$ and $p'= \frac{p}{p-1}$, $p>1,$ we get

		\begin{equation}
			\label{hilbertestimate}
			\|y_k^h\|_{p'}\leq \left(\sum_{n=0}^{\infty}\frac{M^{p'}}{(n+k+1)^{rp'}}\right)^{\frac{1}{p'}}\leq M \left(\int_{k}^{\infty}\frac{1}{x^{rp'}}dx\right)^{\frac{1}{p'}}=C\frac{1}{k^{r-\frac{1}{p'}}},
		\end{equation}
		
		\noindent where $C$ depends on $r$ and $p'$ but not on $k$.

		\noindent
		
		For the Ces\`aro operator we get, for some $C$ which does not depend on $k$,
		
		\begin{equation}
			\label{cesaroestimate}
			\|y_k^c\|_{p'}\leq (k+1)^{\frac{1}{p'}}\mu_k\leq C\frac{1}{k^{r-\frac{1}{p'}}}.
		\end{equation}
		
		\noindent  Hence $(\|y_k^h\|_{p'})_k\in l_q$ and $(\|y_k^c\|_{p'})_k\in l_q$ whenever

		\begin{equation}
			\label{generalestimate}
			q\left(r-\frac{1}{p'}\right)=q\left(r-\frac{p-1}{p}\right)>1.
		\end{equation}

		\noindent Since $r>1$ and $q\geq 1$, the above inequality holds for $p>1$ close enough to 1. Now the result follows directly from Lemma \ref{nuclearcompact}.
		For the statements about $\ell^2_A$, i.e $p=p'=2$, we observe that the estimate \eqref{generalestimate} is satisfied for every $r>1$ if $q=2$. For $q=1$ the estimate is satisfied when $r>\frac32$.
	\end{proof}

	\begin{theorem}
		Let $\mu$ be an $r$-Carleson  measure on $[0,1)$ for some $r>1$.
		\begin{itemize}
			\item[(i)] $H_\mu: H^\infty\to A(\T)$ and $C_\mu: H^\infty\to A(\T)$ are both well defined nuclear operators.
			\item[(ii)] $H_\mu: H^2\to H^2$ and $C_\mu: H^2\to H^2$ are both well defined compact operators. $H_\mu$ is  self-adjoint.
			\item[(iii)] If $r>\frac32,$ then $H_\mu: H^2\to A(\T)$ and $C_\mu: H^2\to A(\T)$ are both well defined nuclear operators.
		\end{itemize}
	\end{theorem}
	\begin{proof}
		From the continuous inclusion $H^\infty\hookrightarrow \ell^p_A$ for each $p>1$, we get $(i)$, considering the composition of the inclusion and $H_\mu: \ell^p_A\to \ell^1_A$ for some convenient $p$ close enough to 1 to ensure that $H_\mu$ is nuclear as a consequence of  Proposition \ref{rcarleson}. The identification $H^2=\ell^2_A,$ together with Proposition \ref{rcarleson}, provides (ii) and (iii). We see $H_\mu$ is self-adjoint proceeding as in the proof of Proposition \ref{wienner}.
	\end{proof}

\vspace{.4cm}	

\textbf{Acknowledgments.}
The authors are grateful to Prof. A. Aleman for his help concerning the convergence of the series defining the generalized Hilbert operator and his suggestion to use Abel summability. They also want to thank Prof. J.A. Pel\'aez for providing them with Example \ref{expel}.

The research of  José Bonet and Enrique Jord\'a was partially supported by the project PID2020-119457GB-100 funded by MCIN/AEI/10.13039/501100011033 and by
``ERDF A way of making Europe''.

\vspace{.5cm}	

\textbf{Statements and Declarations}

\begin{itemize}
	\item 	\textbf{Competing Interests:} The authors have no relevant financial or non-financial interests to disclose.
	
	\item 	\textbf{Availability of data and materials:}  	 Data sharing not applicable to this article as no datasets were generated 	or analysed.
	
	\item 	\textbf{Authors' contributions:} All authors contributed equally to the study conception and design. All authors read and approved the final manuscript.
	
\end{itemize}

\end{document}